\documentclass[11pt]{article}
\usepackage{latexsym,amssymb,amsmath,amsthm,enumerate,geometry,float,cite}
\newtheorem{theorem}{Theorem}

\newtheorem{claim}{Claim}
\usepackage{lineno}
\usepackage{setspace}

\begin{document}
\onehalfspace

\title{Maximizing the Mostar index for bipartite graphs and split graphs}
\author{
\v{S}tefko Miklavi\v{c}$^1$\and 
Johannes Pardey$^2$\and 
Dieter Rautenbach$^2$\and 
Florian Werner$^2$}
\date{}

\maketitle
\vspace{-10mm}
\begin{center}
{\small 
$^1$ University of Primorska, Institute Andrej Maru\v{s}i\v{c}, Koper, Slovenia\\
\texttt{stefko.miklavic@upr.si}\\[3mm]
$^2$ Institute of Optimization and Operations Research, Ulm University, Ulm, Germany\\
\texttt{$\{$johannes.pardey,dieter.rautenbach,florian.werner$\}$@uni-ulm.de}
}
\end{center}

\begin{abstract}
Do\v{s}li\'{c} et al.~defined the Mostar index of a graph $G$ 
as $\sum\limits_{uv\in E(G)}|n_G(u,v)-n_G(v,u)|$,
where, for an edge $uv$ of $G$,
the term $n_G(u,v)$ denotes the number of vertices of $G$
that have a smaller distance in $G$ to $u$ than to $v$.
Contributing to conjectures posed by Do\v{s}li\'{c} et al.,
we show that 
the Mostar index of bipartite graphs of order $n$ is at most $\frac{\sqrt{3}}{18}n^3$, 
and that
the Mostar index of split graphs of order $n$ is at most $\frac{4}{27}n^3$.\\[3mm]
\end{abstract}

\section{Introduction}

Do\v{s}li\'{c} et al.~\cite{domasktizu} defined the {\it Mostar index} $Mo(G)$ of a graph $G$ as
$$Mo(G)=\sum\limits_{uv\in E(G)}|n_G(u,v)-n_G(v,u)|,$$
where, for an edge $uv$ of $G$,
the term $n_G(u,v)$ denotes the number of vertices of $G$
that have a smaller distance in $G$ to $u$ than to $v$.
They conjectured that the complete bipartite graph
$K_{n/3,2n/3}$ has maximum Mostar index among all bipartite graphs of order $n$
(cf. \cite[Conjecture 19]{domasktizu}),
and that
$S_{n/3,2n/3}$ has maximum Mostar index among all graphs of order $n$
(cf. \cite[Conjecture 20]{domasktizu}),
where $S_{k,n-k}$ denotes the split graph that arises from the disjoint
union of a clique $C$ of order $k$ and an independent set $I$ of order $n-k$
by adding all possible edges between $C$ and $I$.
The Mostar index of a complete bipartite graph $K_{\alpha n,(1-\alpha)n}$ 
with $\alpha \leq \frac{1}{2}$ equals $\alpha(1-\alpha)(1-2\alpha)n^3$, 
which, considered as a function of $\alpha$, 
is maximized for $\alpha_1=\frac{1}{2}\left(1-\frac{1}{\sqrt{3}}\right)\approx 0.21132$.
As observed by Geneson and Tsai \cite{gets},
this means that, for sufficiently large $n$, 
Conjecture 19 from \cite{domasktizu} does not hold,
because the Mostar index of $K_{\lfloor \alpha_1 n\rfloor,n-\lfloor \alpha_1 n\rfloor}$ 
is larger than that of $K_{\lfloor n/3\rfloor,n-\lfloor n/3\rfloor}$.

Since its introduction in 2018 
the Mostar index has already incited a lot of research,
mostly concerning sparse graphs and trees  {\cite{AXK, DL2, DL3, HZ1, HZ2, Tepeh}, chemical graphs \cite{CLXZ, DL1, GA, GR, XZTHD}, and hypercube-related graphs \cite{Mollard, OSS}, see also the recent survey \cite{aldo}.
However, the two mentioned conjectures from \cite{domasktizu} remained open,
which suggests that they are difficult.
In fact, they do not seem amenable to the usual transformation arguments 
that modify a given graph iteratively.
An interesting feature of the Mostar index in dense graphs is that,
in order to maximize it,
it seems plausible to have many edges.
If the graphs get too dense though, 
then the contribution of each individual edge to the Mostar index tends to decrease.
Therefore, maximizing the Mostar index over a class of possibly dense graphs 
requires a delicate balance between the total number of edges and 
the average contribution per edge.

In this paper we contribute to the above mentioned conjectures by proving the following two theorems.

\begin{theorem}\label{theorem1}
If $G$ is a bipartite graph of order $n$, then 
$$Mo(G) \leq \alpha_1(1-\alpha_1)(1-2\alpha_1)n^3
=\frac{\sqrt{3}}{18}n^3
\approx 0.096225n^3.$$
\end{theorem}
By the above discussion,
Theorem \ref{theorem1} 
is best possible up to terms of order $O(n^2)$.
In particular, the factor $\frac{\sqrt{3}}{18}$ is best possible.

\begin{theorem}\label{theorem2}
If $G$ is a split graph that arises from a clique $C$ of order $\alpha n$
and an independent set $I$ of order $(1-\alpha)n$ for some $\alpha \in [0,1]$
by adding $m$ edges between vertices in $C$ and vertices in $I$, then 
\begin{eqnarray*}
Mo(G) & \leq & ((1+\alpha)n-1)m-\frac{2m^2}{(1-\alpha)n}\\
&\leq &
\begin{cases}
\alpha(1-\alpha)n^2\Big((1-\alpha)n-1\Big) & 
\mbox{, if $\alpha\leq \frac{1}{3}-\frac{1}{3n}$,}\\[3mm]
\frac{1}{8}(1-\alpha)n\Big((1+\alpha)n-1\Big)^2 &
\mbox{, if $\alpha>\frac{1}{3}-\frac{1}{3n}$}
\end{cases}\\
&\leq & \frac{4}{27}n^3.
\end{eqnarray*}
\end{theorem}
Each bound stated in Theorem \ref{theorem2} 
is best possible up to terms of lower order.
More precisely, for each stated bound, say $b(n,\alpha,m)$,
there are graphs $G$ of arbitrarily large order $n$ 
with $Mo(G)\geq b(n,\alpha,m)-O(n^2)$.

All proofs are given in the next section.

\section{Proofs}

Our proofs rely on relaxations of the Mostar index,
linear programming, duality,
and arguments from optimization theory.

\begin{proof}[Proof of Theorem \ref{theorem1}]
Let $G$ be a bipartite graph of order $n$.

For an edge $uv$ of $G$ with $n_G(u,v)\geq n_G(v,u)$, 
we have 
$n_G(v,u)\geq d_G(v)$ 
and
$n_G(u,v)\leq n-n_G(v,u)\leq n-d_G(v)$,
which implies
$$
|n_G(u,v)-n_G(v,u)|
=n_G(u,v)-n_G(v,u)
\leq n-2d_G(v)
\leq n\left(1-\frac{2}{n}\min\big\{ d_G(u),d_G(v)\big\}\right).$$
For the Mostar index of $G$, this implies
\begin{eqnarray}\label{e1}
Mo(G) &\leq & \sum\limits_{uv\in E(G)}n\left(1-\frac{2}{n}\min\big\{ d_G(u),d_G(v)\big\}\right).
\end{eqnarray}
Let the two partite sets $V_1$ and $V_2$ of $G$ have orders 
$\alpha n$ and $(1-\alpha)n$ for some $0<\alpha\leq \frac{1}{2}$, respectively.
Let 
$I=\{ 0,1,\ldots,(1-\alpha)n\}$ and
$J=\{ 0,1,\ldots,\alpha n\}$.
Let $V_1$ contain exactly $x_i\alpha n$ vertices of degree $i$ for every $i\in I$,
let $V_2$ contain exactly $y_j(1-\alpha) n$ vertices of degree $j$ for every $j\in J$, and
let $G$ have exactly $m_{i,j}\alpha(1-\alpha)n^2$ edges 
between a vertex from $V_1$ of degree $i$ and a vertex from $V_2$ of degree $j$
for every $(i,j)\in I\times J$.
The number of edges in $G$ that are incident with a vertex from $V_1$ of degree $i\in I$ equals 
$\sum\limits_{j\in J}m_{i,j}\alpha(1-\alpha)n^2=i x_i\alpha n$,
which implies 
$$\mbox{$\sum\limits_{j\in J}m_{i,j}-\frac{ix_i}{(1-\alpha)n}=0$ for every $i\in I$.}$$
Symmetrically, we obtain
$$\mbox{$\sum\limits_{i\in I}m_{i,j}-\frac{jy_j}{\alpha n}=0$ for every $j\in J$.}$$
By (\ref{e1}), we obtain 
\begin{eqnarray}\label{e2}
Mo(G)&\leq &\alpha(1-\alpha){\rm OPT}(P)n^3,
\end{eqnarray}
where ${\rm OPT}(P)$ denotes the optimal value 
of the following linear programm $(P)$:
$$
\begin{array}{rrrcll}
& \max  & \sum\limits_{(i,j)\in I\times J}m_{i,j}\left(1-\frac{2}{n}\min\{ i,j\}\right) ,&&&\\
& s.th.  & \sum\limits_{i\in I}x_i&=&1,&\\
(P) \,\,\,\,\,\,\,\,\, &   & \sum\limits_{j\in J}y_j&=&1,&\\
&   & \sum\limits_{j\in J}m_{i,j}-\frac{ix_i}{(1-\alpha)n}&=&0&\mbox{ for every $i\in I$},\\
&   & \sum\limits_{i\in I}m_{i,j}-\frac{jy_j}{\alpha n}&=&0&\mbox{ for every $j\in J$},\\
&   & x_i,y_j,m_{i,j} &\geq &0& \mbox{ for every $(i,j)\in I\times J$}.
\end{array}
$$
The dual of $(P)$ is the following linear programm $(D)$:
$$
\begin{array}{rrrcll}
& \min & p+q, &  &  &\\
& s.th. & p_i+q_j & \geq & 1-\frac{2}{n}\min\{ i,j\} & \mbox{ for every $i\in I$ and every $j\in J$},\\
(D) \,\,\,\,\,\,\,\,\, & & p & \geq & \frac{ip_i}{(1-\alpha)n} & \mbox{ for every $i\in I$},\\
& & q & \geq & \frac{jq_j}{\alpha n} & \mbox{ for every $j\in J$},\\
& & p,q,p_j,q_j & \in & \mathbb{R} & \mbox{ for every $i\in I$ and every $j\in J$}.
\end{array}
$$
For our argument, we actually only need the 
{\it weak duality inequality chain} for $(P)$ and $(D)$:
\begin{eqnarray}
&& \sum\limits_{(i,j)\in I\times J}m_{i,j}\left(1-\frac{2}{n}\min\{ i,j\}\right)\nonumber\\
& \leq &
\sum\limits_{(i,j)\in I\times J}m_{i,j}(p_i+q_j)
+\sum\limits_{i\in I}x_i\left(p-\frac{ip_i}{(1-\alpha)n}\right)
+\sum\limits_{j\in J}y_j\left(q-\frac{jq_j}{\alpha n}\right)\label{e3}\\
& = &
\sum\limits_{i\in I}x_ip
+\sum\limits_{j\in J}y_jq
+\sum\limits_{i\in I}\left(\sum\limits_{j\in J}m_{i,j}-\frac{ix_i}{(1-\alpha)n}\right)p_i
+\sum\limits_{j\in J}\left(\sum\limits_{i\in I}m_{i,j}-\frac{jy_j}{\alpha n}\right)q_j\label{e4}\\
& = & p+q,\label{e5}
\end{eqnarray}
where 
(\ref{e3}) follows from the non-negativity of the primal variables $x_i$, $y_j$, and $m_{i,j}$ and the conditions within $(D)$,
(\ref{e4}) corresponds to reordering terms, and
(\ref{e5}) follows from the conditions within $(P)$.
In particular, we have ${\rm OPT}(P)\leq {\rm OPT}(D)$.

We consider the following non-linear optimization problem $(D')$:
$$
\begin{array}{rrrcll}
& \min & p+q ,&  &  &\\
& s.th. & p+q & \geq & 1-2\alpha, &\\
(D') \,\,\,\,\,\,\,\,\, & & p+2\sqrt{2q\alpha} & \geq & 1 & \mbox{ if $q<2\alpha$},\\
& & 2\sqrt{2p(1-\alpha)}+q & \geq & 1 & \mbox{ if $p<\frac{2\alpha^2}{1-\alpha}$},\\
& & p,q & \geq & 0. &
\end{array}
$$
Claim \ref{claim1} below links $(D)$ and $(D')$.
Its proof uses the following elementary observation:
The function 
$$f:(0,\infty)\to \mathbb{R}:x\mapsto \frac{\beta}{x}+\gamma x$$ 
for $\beta\geq 0$ and $\gamma>0$ is 
non-increasing for $x\leq \sqrt{\frac{\beta}{\gamma}}$ and 
non-decreasing for $x\geq \sqrt{\frac{\beta}{\gamma}}$.
For some $\delta > 0$, this implies
\begin{eqnarray}\label{e6}
\min\Big\{ f(x):x\in (0,\delta]\Big\} & = & 
\begin{cases}
f\left(\sqrt{\frac{\beta}{\gamma}}\right)=2\sqrt{\beta\gamma} & \mbox{, if $\delta\geq \sqrt{\frac{\beta}{\gamma}}$},\\
f\left(\delta\right) & \mbox{, if $\delta\leq \sqrt{\frac{\beta}{\gamma}}$}.
\end{cases}
\end{eqnarray}
\begin{claim}\label{claim1}
${\rm OPT}(D)\leq {\rm OPT}(D')$.
\end{claim}
\begin{proof}[Proof of Claim \ref{claim1}]
Let $(p,q)$ be a feasible solution of $(D')$.
Let 
$p_0=q_0=1$,
$p_i=\frac{(1-\alpha)np}{i}$ for $i\in I\setminus \{ 0\}$, and
$q_j=\frac{\alpha nq}{j}$ for $j\in J\setminus \{ 0\}$.
In order to complete the proof, it suffices to show that
$p$, $q$, the $p_i$s, and the $q_j$s form a feasible solution of $(D)$.
By definition, we have
$p \geq \frac{ip_i}{(1-\alpha)n}$ for every $i\in I$
and
$q \geq \frac{jq_j}{\alpha n}$ for every $j\in J$.

Now, let $(i,j)\in I\times J$.

First, suppose that $\min\{ i,j\}=0$.
In this case, we have $p_i+q_j\geq 1\geq 1-\frac{2}{n}\min\{i,j\}$.

Next, suppose that $1\leq j\leq i$.
The inequality
$p+\frac{\alpha nq}{j}\geq 1-\frac{2j}{n}$
is equivalent to
$p+\left(\frac{q}{x}+2\alpha x\right)\geq 1$ for $x=\frac{j}{\alpha n}\in (0,1]$.
If $q\geq 2\alpha$, then $1\leq \sqrt{\frac{q}{2\alpha}}$, and
$p+\left(\frac{q}{x}+2\alpha x\right)
\stackrel{(\ref{e6})}{\geq} p+\left(q+2\alpha\right)
\stackrel{(D')}{\geq} 1$.
If $q<2\alpha$, then $1>\sqrt{\frac{q}{2\alpha}}$, and 
$p+\left(\frac{q}{x}+2\alpha x\right)
\stackrel{(\ref{e6})}{\geq} p+2\sqrt{2q\alpha}
\stackrel{(D')}{\geq} 1$.
Altogether, it follows that $p+\frac{\alpha nq}{j}\geq 1-\frac{2j}{n}$, and, hence,
$$p_i+q_j=\frac{(1-\alpha)np}{i}+\frac{\alpha nq}{j}
\stackrel{i\leq (1-\alpha) n}{\geq} p+\frac{\alpha nq}{j}\geq 1-\frac{2j}{n}=1-\frac{2}{n}\min\{ i,j\}.$$
Finally, suppose that $1\leq i\leq j$.
Since $j\leq \alpha n$, we have $i\leq \alpha n$ in this case.
The inequality
$\frac{(1-\alpha)np}{i}+q\geq 1-\frac{2i}{n}$
is equivalent to
$\left(\frac{p}{x}+2(1-\alpha) x\right)+q\geq 1$ for $x=\frac{i}{(1-\alpha) n}\in \left(0,\frac{\alpha}{1-\alpha}\right]$.
If $p\geq \frac{2\alpha^2}{1-\alpha}$, 
then $\frac{\alpha}{1-\alpha}\leq \sqrt{\frac{p}{2(1-\alpha)}}$, and
$\left(\frac{p}{x}+2(1-\alpha) x\right)+q
\stackrel{(\ref{e6})}{\geq}
\left(\frac{1-\alpha}{\alpha}p+2\alpha\right)+q
\stackrel{\alpha\leq 1/2}{\geq}
p+2\alpha+q
\stackrel{(D')}{\geq} 1$.
If $p<\frac{2\alpha^2}{1-\alpha}$, 
then $\frac{\alpha}{1-\alpha}>\sqrt{\frac{p}{2(1-\alpha)}}$, and
$\left(\frac{p}{x}+2(1-\alpha) x\right)+q
\stackrel{(\ref{e6})}{\geq}
2\sqrt{2p(1-\alpha)}+q 
\stackrel{(D')}{\geq} 1$.
Altogether, it follows that $\frac{(1-\alpha)np}{i}+q\geq 1-\frac{2i}{n}$, and, hence,
$$p_i+q_j=\frac{(1-\alpha)np}{i}+\frac{\alpha nq}{j}
\stackrel{j\leq \alpha n}{\geq} \frac{(1-\alpha)np}{i}+q\geq 1-\frac{2i}{n}=1-\frac{2}{n}\min\{ i,j\}.$$
Since all conditions within $(D)$ are satisfied, it follows that
$p$, $q$, the $p_i$s, and the $q_j$s form a feasible solution of $(D)$
as desired, which completes the proof of the claim.
\end{proof}
Combining (\ref{e2}), the weak duality inequality ${\rm OPT}(P)\leq {\rm OPT}(D)$, and Claim \ref{claim1}
implies
$$Mo(G)\leq \alpha(1-\alpha){\rm OPT}(D')n^3,$$
and the following claim completes the proof.

\begin{claim}\label{claim2}
${\rm OPT}(D')\leq \frac{\alpha_1(1-\alpha_1)(1-2\alpha_1)}{\alpha(1-\alpha)}$.
\end{claim}
\begin{proof}[Proof of Claim \ref{claim2}]
Since $(D')$ is a minimization problem, 
it suffices to provide a feasible solution $(p,q)$ for $(D')$
such that $p+q$ is at most the desired upper bound on ${\rm OPT}(D')$.

First, suppose that $\alpha\leq \alpha_2:=\frac{5-\sqrt{17}}{4}\approx 0.21922$,
which implies $2\alpha+\frac{2\alpha^2}{1-\alpha}\leq 1-2\alpha$.
In this case, we choose $q\geq 2\alpha$ and $p\geq \frac{2\alpha^2}{1-\alpha}$
in such a way that $p+q=1-2\alpha$,
which yields a feasible solution for $(D')$.
As observed in the introduction,
the function $x\mapsto x(1-x)(1-2x)$ with $x\in \left[0,\frac{1}{2}\right]$
is maximized for $x=\alpha_1$,
which implies 
$\alpha(1-\alpha)(p+q)=\alpha(1-\alpha)(1-2\alpha)\leq \alpha_1(1-\alpha_1)(1-2\alpha_1)$, and, hence,
$p+q\leq \frac{\alpha_1(1-\alpha_1)(1-2\alpha_1)}{\alpha(1-\alpha)}$.

Next, suppose that $\alpha\geq \alpha_2$.
In this case, we choose $p=0.42\alpha$ and $q=\frac{0.09622}{\alpha(1-\alpha)}-p$.

Since $p+q=\frac{0.09622}{\alpha(1-\alpha)}<
\frac{\alpha_1(1-\alpha_1)(1-2\alpha_1)}{\alpha(1-\alpha)}$,
it remains to show that $(p,q)$ is a feasible solution for $(D')$.

For $q=q(\alpha)=\frac{0.09622}{\alpha(1-\alpha)}-0.42\alpha$,
we obtain 
$$q'(\alpha)=-\frac{0.09622}{\alpha^2(1-\alpha)}+\frac{0.09622}{\alpha(1-\alpha)^2}-0.42,$$ 
which is negative for $\alpha\in [\alpha_2,0.5]$.
This implies that $q(\alpha)\geq q(0.5)=0.17488>0$ for $\alpha\in [\alpha_2,0.5]$.

For $f_1(\alpha)=p+q-(1-2\alpha)=\frac{0.09622}{\alpha(1-\alpha)}-1+2\alpha$,
we obtain 
$$f_1''(\alpha)=\frac{2\cdot 0.09622}{\alpha^3(1-\alpha)}-\frac{2\cdot 0.09622}{\alpha^2(1-\alpha)^2}+\frac{2\cdot 0.09622}{\alpha(1-\alpha)^3},$$ 
which is positive for $\alpha\in [\alpha_2,0.5]$.
This implies that 
$f_1'(\alpha)\geq f_1'(\alpha_2)>0.15571>0$ for $\alpha\in [\alpha_2,0.5]$,
and, hence,
$f_1(\alpha)\geq f_1(\alpha_2)>0.00059>0$ for $\alpha\in [\alpha_2,0.5]$.
It follows that $p+q\geq 1-2\alpha$.

For $f_2(\alpha)=p+2\sqrt{2q\alpha}-1$, we obtain
$$f_2'(\alpha)=0.42-\sqrt{2}\frac{2\cdot 0.42\alpha-\frac{0.09622}{(1-\alpha)^2}}{\sqrt{\left(\frac{0.09622}{\alpha(1-\alpha)}-0.42\alpha\right)\alpha}}.$$
On the interval $[\alpha_2,0.5]$,
the term $2\cdot 0.42\alpha-\frac{0.09622}{(1-\alpha)^2}$ is maximized 
for $\alpha=1-\left(\frac{0.09622}{0.42}\right)^{\frac{1}{3}}$,
in which case it is strictly smaller than $0.06903$.
Similarly,
on the interval $[\alpha_2,0.5]$,
the term $\sqrt{\left(\frac{0.09622}{\alpha(1-\alpha)}-0.42\alpha\right)\alpha}$ is minimized 
for $\alpha=0.5$,
in which case it is strictly larger than $0.29570$.
This implies that 
$f_2'(\alpha)>0.42-\sqrt{2}\frac{0.06903}{0.29570}>0.08985>0$ 
for $\alpha\in [\alpha_2,0.5]$,
and, hence,
$f_2(\alpha)\geq f_2(\alpha_2)>0.00004>0$ for $\alpha\in [\alpha_2,0.5]$.
It follows that $p+2\sqrt{2q\alpha}\geq 1$.

For $f_3(\alpha)=2\sqrt{2p(1-\alpha)}+q-1
=2\sqrt{2\cdot 0.42}\sqrt{\alpha(1-\alpha)}+\frac{0.09622}{\alpha(1-\alpha)}-0.42\alpha-1$, we obtain
$$f_3'(\alpha)=
\frac{\sqrt{2\cdot 0.42}(1-2\alpha)}{\sqrt{\alpha(1-\alpha)}}
-\frac{0.09622}{\alpha^2(1-\alpha)}
+\frac{0.09622}{\alpha(1-\alpha)^2}
-0.42.$$
Similarly as above, it is a routine task to verify that 
$f_3'(\alpha)$ is negative on $[\alpha_2,0.5]$,
which implies
$f_3(\alpha)\geq f_3(0.5)>0.09139>0$
for $\alpha\in [\alpha_2,0.5]$.
It follows that $2\sqrt{2p(1-\alpha)+q}\geq 1$,
which completes the proof of the claim.
\end{proof}
As observed above, this complete the proof.
\end{proof}

\begin{proof}[Proof of Theorem \ref{theorem2}]
Let $G$ be a split graph that arises from a clique $C$ of order $\alpha n$
and an independent set $I$ of order $(1-\alpha)n$ for some $\alpha \in [0,1]$
by adding $m$ edges between vertices in $C$ and vertices in $I$.

For an edge $uv$ of $G$ with $u\in C$ and $v\in I$, we have $n_G(v,u)=1$ and
$n_G(u,v)\leq |I\setminus \{ v\}|+|\{ u\}|+|C\setminus N_G(v)|=n-d_G(v)$,
which implies
$$|n_G(u,v)-n_G(v,u)|\leq n-d_G(v)-1.$$
Let $E$ be the set of the $m$ edges of $G$ between $C$ and $I$.

For an edge $uu'$ of $G$ with $u,u'\in C$, we have 
$n_G(u,u')=|(N_G(u)\cap I)\setminus (N_G(u')\cap I)|$
and
$n_G(u',u)=|(N_G(u')\cap I)\setminus (N_G(u)\cap I)|$,
which implies
$$|n_G(u,u')-n_G(u',u)|=|d_G(u)-d_G(u')|.$$
For the Mostar index of $G$, this implies
\begin{eqnarray}
Mo(G) 
&\leq & 
\sum\limits_{uv\in E}(n-d_G(v)-1)
+\sum\limits_{uu'\in {C\choose 2}}|d_G(u)-d_G(u')|\nonumber \\
&=& 
m(n-1)-\sum\limits_{v\in I}d_G(v)^2
+\sum\limits_{uu'\in {C\choose 2}}|d_G(u)-d_G(u')|\nonumber \\
& \leq & 
m(n-1)-\frac{m^2}{(1-\alpha)n}
+\sum\limits_{uu'\in {C\choose 2}}|d_G(u)-d_G(u')|,\label{e7}
\end{eqnarray}
where the final inequality uses the Cauchy-Schwarz inequality
$$\sum\limits_{v\in I}d_G(v)^2
\geq \frac{1}{|I|}\left(\sum\limits_{v\in I}d_G(v)\right)^2
\geq \frac{m^2}{(1-\alpha)n}.$$
\begin{claim}\label{claim3}
$\sum\limits_{uu'\in {C\choose 2}}|d_G(u)-d_G(u')|\leq \alpha n m-\frac{m^2}{(1-\alpha)n}$.
\end{claim}
\begin{proof}[Proof of Claim \ref{claim3}]
Let $m=(1-\alpha)nr+s$ for non-negative integers $r$ and $s$ with $s<(1-\alpha)n$.

Let the $\alpha n$ vertices in $C$ have 
$d_1\geq d_2\geq \ldots \geq d_{\alpha n}$ 
neighbors in $I$, respectively.

Note that 
\begin{eqnarray*}
\sum\limits_{uu'\in {C\choose 2}}|d_G(u)-d_G(u')|
&=& \sum\limits_{1\leq i<j\leq \alpha n}(d_i-d_j)\\
&=& \Big(\alpha n+1-2\cdot 1\Big)d_1
+\Big(\alpha n+1-2\cdot 2\Big)d_2
+\cdots+
\Big(\alpha n+1-2\cdot \alpha n\Big)d_{\alpha n}.
\end{eqnarray*}
Since each $d_i$ is at most $(1-\alpha)n$ and their sum is $m$,
this term is maximized for $d_1=\ldots=d_r=(1-\alpha)n$ and $d_{r+1}=s$, 
and we obtain
\begin{eqnarray*}
\sum\limits_{uu'\in {C\choose 2}}|d_G(u)-d_G(u')|
&\leq & \sum\limits_{i=1}^{r}\Big(\alpha n+1-2\cdot i\Big)(1-\alpha)n
+\Big(\alpha n+1-2\cdot (r+1)\Big)s\\
&=& (\alpha n-r)r(1-\alpha)n+(\alpha n-1-2r)s\\
&\leq &\alpha n m-\frac{m^2}{(1-\alpha)n},
\end{eqnarray*}
where the final inequality holds, since, using $m=(1-\alpha)nr+s$,
we obtain
$$\left(\alpha n m-\frac{m^2}{(1-\alpha)n}\right)-
\Big((\alpha n-r)r(1-\alpha)n+(\alpha n-1-2r)s\Big)
=\frac{s((1-\alpha)n-s)}{(1-\alpha)n}\geq 0.$$
\end{proof}
By (\ref{e7}) and Claim \ref{claim3}, we obtain
$$
Mo(G) \leq 
m(n-1)-\frac{m^2}{(1-\alpha)n}
+\alpha n m-\frac{m^2}{(1-\alpha)n}
=((1+\alpha)n-1)m-\frac{2m^2}{(1-\alpha)n}:=g(n,\alpha,m).$$
For fixed $n$ and $\alpha$, 
the expression $g(n,\alpha,m)$ is a quadratic function of $m$,
which is non-decreasing for $m\leq m^*:=\frac{1}{4}(1-\alpha)n\Big((1+\alpha)n-1\Big)$
and non-increasing for $m\geq m^*$.

If $\alpha\leq\frac{1}{3}-\frac{1}{3n}$, 
then $m\leq |C|\cdot |I|=\alpha(1-\alpha)n^2\leq m^*$, and 
$$Mo(G) \leq g(n,\alpha,m)\leq g\left(n,\alpha,\alpha(1-\alpha)n^2\right)
=\alpha(1-\alpha)n^2\Big((1-\alpha)n-1\Big).$$
If $\alpha>\frac{1}{3}-\frac{1}{3n}$, 
then 
$$Mo(G) \leq g(n,\alpha,m)\leq g\left(n,\alpha,m^*\right)
=\frac{1}{8}(1-\alpha)n\Big((1+\alpha)n-1\Big)^2.$$
Since 
$\max\left\{\alpha(1-\alpha)^2,\frac{1}{8}(1-\alpha)(1+\alpha)^2\right\}\leq \frac{4}{27}$
for every $\alpha\in [0,1]$,
the proof is complete.
\end{proof}
In order to construct the graphs mentioned after Theorem \ref{theorem2},
showing that the stated bounds are best possible up to terms of lower order,
one can follow the above proof:
The $m$ edges between $C$ and $I$ should be arranged in such a way that 
the degrees of the vertices in $I$ are as regular as possible 
(reducing the error in the application of the Cauchy-Schwarz inequality)
and $C$ contains as many universal vertices as possible
(reducing the error in the inequality from Claim \ref{claim3}). 
This actually corresponds exactly to the situation 
$d_1=\ldots=d_r=(1-\alpha)n$ and $d_{r+1}=s$
considered in the proof of Claim \ref{claim3}.

\end{document}